\documentclass[12pt,reqno]{amsart}
\usepackage[all]{xy}
\usepackage[headings]{fullpage}
\usepackage{times}

\numberwithin{subsection}{section}

\numberwithin{equation}{section}

\theoremstyle{plain}
\newtheorem{thm}[equation]{Theorem}
\newtheorem{prop}[equation]{Proposition}
\newtheorem{cor}[equation]{Corollary}

\theoremstyle{definition}
\newtheorem{defn}[equation]{Definition}

\theoremstyle{remark}

\newtheorem{rem-exer}[equation]{Remark/Exercise}
\newtheorem{rem-exers}[equation]{Remark/Exercises}

\usepackage[OT2,T1]{fontenc}
\DeclareSymbolFont{cyrletters}{OT2}{wncyr}{m}{n}
\DeclareMathSymbol{\sha}{\mathalpha}{cyrletters}{"58}

\newcommand{\CC}{\mathcal{C}}

\newcommand{\DD}{\mathcal{D}}

\renewcommand{\O}{\mathcal{O}}

\newcommand{\EE}{\mathcal{E}}

\newcommand{\NN}{\mathcal{N}}

\newcommand{\F}{\mathbb{F}}
\newcommand{\Fp}{{\mathbb{F}_p}}

\newcommand{\Fq}{{\mathbb{F}_q}}

\newcommand{\Fpbar}{{\overline{\mathbb{F}}_p}}

\newcommand{\Qbar}{{\overline{\mathbb{Q}}}}

\newcommand{\ratto}{{\dashrightarrow}}

\newcommand{\Z}{\mathbb{Z}}

\newcommand{\Q}{\mathbb{Q}}

\renewcommand{\P}{\mathbb{P}}

\newcommand{\p}{\mathfrak{p}}

\newcommand{\into}{\hookrightarrow}

\newcommand{\tensor}{\otimes}
\newcommand{\compose}{\circ}

\newcommand{\labeledto}[1]{\overset{#1}{\to}}

\DeclareMathOperator{\Ker}{Ker}
\DeclareMathOperator{\coker}{Coker}

\DeclareMathOperator{\ord}{ord}
\DeclareMathOperator{\rk}{Rank}

\DeclareMathOperator{\gal}{Gal}
\DeclareMathOperator{\spec}{Spec}

\DeclareMathOperator{\disc}{Disc}

\DeclareMathOperator{\Sel}{Sel}
\DeclareMathOperator{\Fr}{Fr}

\def\clap#1{\hbox to 0pt{\hss#1\hss}}

\renewcommand{\NN}{{Nm_q}}

\begin{document}
\title{Explicit points on the Legendre curve II}

\author{Ricardo P. Concei\c c\~ao}
\address{Oxford College of Emory University  \\ Oxford, GA 30054}
\email{rconcei@emory.edu}
\author{Chris Hall}
\address{Department of Mathematics \\ University of Wyoming 
  \\ Laramie, WY 82071}
\email{chall14@uwyo.edu}
\author{Douglas Ulmer}
\address{School of Mathematics \\ Georgia Institute of Technology
  \\ Atlanta, GA 30332}
\email{ulmer@math.gatech.edu}



\begin{abstract}
  Let $E$ be the elliptic curve $y^2=x(x+1)(x+t)$ over the field
  $\Fp(t)$ where $p$ is an odd prime.  We study the arithmetic of $E$
  over extensions $\Fq(t^{1/d})$ where $q$ is a power of $p$ and $d$
  is an integer prime to $p$.  The rank of $E$ is given in terms of an
  elementary property of the subgroup of $(\Z/d\Z)^\times$ generated
  by $p$.  We show that for many values of $d$ the rank is large.  For
  example, if $d$ divides $2(p^f-1)$ and $2(p^f-1)/d$ is odd, then the
  rank is at least $d/2$.  When $d=2(p^f-1)$, we exhibit explicit
  points generating a subgroup of $E(\Fq(t^{1/d}))$ of finite index in
  the ``2-new'' part, and we bound the index as well as the order of
  the ``2-new'' part of the Tate-Shafarevich group.
\end{abstract}

\maketitle

\section{Introduction}
Fix an odd prime $p$ and consider the elliptic curve
\begin{equation}\label{eq:E}
E:\qquad y^2=x(x+1)(x+t)
\end{equation}
defined over the rational function field $\Fp(t)$.  We call $E$ the
Legendre curve.  In \cite{Ulmer14a}, the third author considered
the arithmetic of $E$ over extensions of the form $\Fq(t^{1/d})$ where
$q$ is a power of $p$ and $d$ is a positive integer prime to $p$,
proving in particular that the Birch and Swinnerton-Dyer conjecture
holds for $E$ over each of these fields.

Most of \cite{Ulmer14a} emphasizes the case when $d$ divides
$p^f+1$ for some $f$.  (These integers $d$ can also be described as
those such that $-1$ is in the cyclic subgroup of $(\Z/d\Z)^\times$
generated by $p$.)  The emphasis on this case is natural because the
integers $p^f+1$ are known to play a special role in arithmetic
geometry in characteristic $p$.  For example, the Jacobian of the
Fermat curve of degree $d$ is supersingular if and only if $d$ divides
$p^f+1$ for some $f\ge1$ \cite{ShiodaKatsura79}.

One of our aims in this note is to point out that that there are many
other values of $d$ such that the curve $E$ has interesting arithmetic
(e.g., large Mordell-Weil rank) over $\Fq(t^{1/d})$.  To that end, we
calculate the Hasse-Weil $L$-function of $E$ over the extensions
$\Fq(t^{1/d})$ for all $q$ and $d$, and we give a simple and explicit
formula for the rank of $E(\Fq(t^{1/d}))$.  Our formula for the rank
is in terms of a new, elementary notion of ``balanced subgroup of the
multiplicative group.''

Using the notion of balanced subgroup, we prove that if $4$ divides
$d$, $p^f\equiv d/2+1\pmod d$ for some $f$, and $q\equiv1\pmod d$,
then $\ord_{s=1}L(E/\Fq(t^{1/d}),s)\ge d/2$ and $E(\Fq(t^{1/d}))$ has
rank at least $d/2$.  These integers $d$ can also be described as
those dividing $2(p^f-1)$ for some $f$ with $2(p^f-1)/d$ odd.  Note
that such integers $d>4$ do not divide $p^{g}+1$ for any value of $g$,
so this is indeed a class of integers distinct from those studied in
\cite{Ulmer14a}.  In \cite{PomeranceUlmer13}, it is shown that
this class of integers is, in a precise sense, more numerous than the
class of divisors of elements of $\{p^g+1|g\ge1\}$.

The statement of our rank formula is in Section~\ref{s:rank}.  In
Section~\ref{s:L-function}, we calculate the $L$-function of $E$ over
$\Fq(t^{1/d})$ for all $d$ and all $q$ in terms of Jacobi sums, and in
Section~\ref{s:Jacobi}, we make these Jacobi sums explicit for those
values of $d$ where the $L$-function has zeroes.  This leads to the
proof of the rank formula stated in Section~\ref{s:rank}.

In Section~\ref{s:balanced}, we make several comments on the condition
that $p$ is balanced modulo $d$, and in Subsection~\ref{ss:d/2+1}, we
prove the rank assertion made above for values $d$ where some power of
$p$ is congruent to $d/2+1$ modulo $d$.

In Sections~\ref{s:points} through \ref{s:sha}, we study the case
$d=2(p^f-1)$ in more detail, exhibiting explicit points which generate
a finite index subgroup of the ``2-new'' part of the Mordell-Weil
group of $E$.  We also obtain bounds on the ``2-new'' part of the
Tate-Shafarevich group of $E$, showing in particular that it is a
$p$-group.

In Section~\ref{s:Berger}, we show how certain endomorphisms of curves
``explain'' the explicit points of \cite{Ulmer14a} and this
paper, in the style of \cite{Ulmer13a}.  Finally, in
Section~\ref{s:p=2}, we explain how many of the results of this paper
can be extended to a context where $p=2$ is relevant.

We thank Ernie Croot and Don Zagier for helpful comments made in the
early stages of this work.  We also thank Carl Pomerance for his
encouragement and useful comments.

\subsection{Notation} For the convenience of the reader, we gather here
notation which is used in several parts of the paper.
\begin{itemize}
\item For a finite set $S$, we write $|S|$ for the cardinality of $S$.
\item We assume that $p$ is an odd prime number and $q$ is
  a power of $p$.  We write $\Fp$ and $\Fq$ for the fields of
  cardinality $p$ and $q$ respectively.  See Section~\ref{s:p=2} for a
  situation where $p=2$ is relevant.
\item Throughout, $d$ will be an integer $>2$ and relatively prime to
  $p$.  See \cite[Rmk.~12.2]{Ulmer14a} for the case when $p$
  divides $d$.
\item Let $K$ be the function field $\Fq(t^{1/d})=\Fq(u)$ where
  $u^d=t$.
\item Throughout, $E$ will denote the elliptic curve over $\Fp(t)$
  defined by Equation~\eqref{eq:E}.  We write $L(E/K,T)$ and
  $L(E/K,s)$ for the Hasse-Weil $L$-function of $E$ over the extension
  $K$ of $\Fp(t)$.  Here $T$ and $s$ are related by $T=q^{-s}$.  We
  write $\sha(E/K)$ for the Tate-Shafarevich group of $E$ over $K$.
\item Let $G=(\Z/d\Z)^\times$ and let $A$ (resp.~$B$) be the subset of
  $G$ consisting of classes whose least positive residue lies in
  $(0,d/2)$ (resp.~$(d/2,d)$).
\item For an integer $a$ prime to $d$, let $\langle a\rangle_d$ be the
  cyclic subgroup of $G$ generated by $a$.  We will only use this
  notation when $a=p$ or $a=q$.  We write $\phi(d)$ for $|G|$ and
  $o_d(q)$ for $|\langle q\rangle_d|$, i.e., for the order of $q$
  modulo $d$.  We will also write $\phi(e)=|(\Z/e\Z)^\times|$ and
  $o_e(q)=|\langle q\rangle_e|$ for divisors $e$ of $d$.
\item We write $o\subset\Z/d\Z$ for an orbit of the action of $\langle
  p\rangle_d$ on $\Z/d\Z$.  We always assume that $o\neq\{0\}$ and
  that $o\neq\{d/2\}$ when $d$ is even.
\item For an orbit $o\subset\Z/d\Z$ as above, we define a certain
  Jacobi sum $J_o$ in Subsection~\ref{ss:Jacobi}
\item We write $\mu_e$ and $\zeta_e$ for the group of roots of unity
  of order $e$ and a fixed primitive root of unity of order $e$
  respectively.  Depending on the context, these are to be taken in
  $\Qbar$, the algebraic closure of $\Q$, or $\Fpbar$, the algebraic
  closure of $\Fp$.  In particular, $\Q(\mu_e)$ denotes the splitting
  field of $x^e-1$ over $\Q$.
\item In Sections~\ref{s:points} through \ref{s:sha}, we take
  $d=2(p^f-1)$ and $q\equiv1\pmod d$ and write down explicit points
  $R_i\in E(\Fq(t^{1/d}))$. The subgroup they generate is denoted
  $W_d$. 
\end{itemize}

\section{The rank of $E$ over $\Fq(t^{1/d})$}\label{s:rank}
We are going to compute the rank of the Mordell-Weil group of $E$ over
$K=\Fq(t^{1/d})$ in terms of a certain property of the subgroup
$\langle p\rangle_e$ of $(\Z/e\Z)^\times$ for divisors $e$ of $d$.

Consider the multiplicative group $G=(\Z/d\Z)^\times$, and recall that
$A\subset G$ is the subset of classes whose least positive residue
lies in $(0,d/2)$, and $B\subset G$ is the subset whose least positive
residue lies in $(d/2,d)$, so that $G$ is the disjoint union of $A$
and $B$.

\begin{defn}
  We say ``$p$ is balanced modulo $d$'' if every coset $C$ of $\langle p
  \rangle_d$ in $G$ satisfies $|C\cap A|=|C\cap B|$.  In other
  words, the cosets of $\langle p \rangle_d$ are all evenly divided
  between the two halves $A$ and $B$ of $G$.  In this case, we also
  say that ``$\langle p\rangle_d$ is a balanced subgroup of $G$.''
\end{defn}

For example, it is easy to see that $p$ is balanced modulo $d$ if
$d>2$ and $-1\in\langle p\rangle_d$, in other words, if $d$ divides
$p^f+1$ for some $f>0$.  See Section~\ref{s:balanced} for
several elementary properties of the balanced condition.  See also
\cite{PomeranceUlmer13} for further analysis of this condition.  

One of our main results is the following calculation of the rank of
the Legendre curve $E$ over the fields $\Fq(t^{1/d})$.

\begin{thm}\label{thm:rank}
  Let $p$ be an odd prime number, let $q$ be a power of $p$, and let
  $d$ be a positive integer not divisible by $p$.  Let $K$ be the
  field $\Fq(t^{1/d})$ and let $E$ be the elliptic curve defined by
  Equation~\eqref{eq:E}.  Then the order of vanishing at $s=1$ of
  $L(E/K,s)$ and the rank of the Mordell-Weil group $E(K)$ are both
  equal to
$$\sum_{\substack{e|d\\e>2}}\begin{cases}
\frac{\phi(e)}{o_e(q)}&\text{if $p$ is balanced modulo $e$}\\
0&\text{if not}
\end{cases}$$
where $\phi$ is Euler's function and $o_e(q)$ is the order of $q$ in
$(\Z/e\Z)^\times$. 
\end{thm}

The proof will be given at the end of Section~\ref{s:Jacobi} below.
The case when $d$ divides $p^f+1$ was already treated in
\cite[Cor.~5.3]{Ulmer14a}.

\section{The $L$-function}\label{s:L-function}
\numberwithin{equation}{subsection}

Our goal in this section is to calculate the $L$-function
of $E$ in terms of Jacobi sums.

\subsection{Jacobi sums}\label{ss:Jacobi}
We review some elementary facts about Jacobi sums to set notation.
These facts can all be found, for example, in \cite{CohenNT1}
although our notations are not identical to those of \cite{CohenNT1}.

Let $\O_{\Qbar}$ be the ring of algebraic numbers and fix a prime
$\p\subset\O_{\Qbar}$ over $p$.  Then $\O_{\Qbar}/\p$ is an algebraic
closure of $\Fp$ which we denote $\Fpbar$.  All finite fields of
characteristic $p$ will be taken to be subfields of this $\Fpbar$.

Reduction modulo $\p$ defines an isomorphism from the group of roots
of unity in $\Qbar$ of order prime to $p$ to $\Fpbar^\times$.  Let
$t:\Fpbar^\times\to\Qbar^\times$ be the inverse of this homomorphism.
The restriction of $t$ to any finite field
$k^\times\subset\Fpbar^\times$ is a character of order $|k^\times|$.
We extend any character $\chi$ of $k^\times$ to $k$ by setting
$\chi(0)=0$ if $\chi$ is non-trivial and $\chi_{triv}(0)=1$.

If $k$ is a finite field and $\chi_1$ and $\chi_2$ are multiplicative
characters $k^\times\to\Qbar^\times$, we define a Jacobi sum
as follows:
$$J(\chi_1,\chi_2)=\sum_{u+v+1=0}\chi_1(u)\chi_2(v)$$
where the sum is over elements $u$ and $v\in k$.  The sum is an
algebraic integer, and when $\chi_1$, $\chi_2$, and $\chi_1\chi_2$ are
all non-trivial, it has absolute value $|k|^{1/2}$ in every complex
embedding, i.e., it is a Weil integer of size $|k|^{1/2}$.

Let $q$ be a fixed power of $p$ and let $d$ be a positive integer not
divisible by $p$.  Multiplication by $q$ defines a permutation of
$\Z/d\Z$ and we consider the orbits of this action.  Let
$o\subset\Z/d\Z$ be such an orbit; we always assume that $o\neq\{0\}$
and, if $d$ is even, that $o\neq\{d/2\}$.

We are going to define a Jacobi sum $J_o$.  To that end, let $|o|$ be
the cardinality of the orbit $o$ and let $i\in o$.  Note that if
$e=d/\gcd(d,i)$, then $|o|$ is the smallest positive integer $f$ such
that $q^{f}\equiv1\pmod e$.  
We set
$$\chi_i=t^{\frac{q^{|o|}-1}{d}i}$$ 
which we view by restriction as a character of $\F_{q^{|o|}}^\times$.
Also, we denote by $\lambda$ the non-trivial quadratic character of
$\F_{q^{|o|}}^\times$.  With these notations, we set
$$J_o=J(\lambda,\chi_i)=\sum_{u\in\F_{q^{|o|}}}\lambda(u)\chi_i(-1-u).$$
This is independent of the choice of $i$ because
$J(\lambda,\chi_{qi})=J(\lambda^q,\chi_i^q)=J(\lambda,\chi_i)$.  Since we
assumed $o\neq\{0\}$ and $\{d/2\}$, the character $\chi_i$ has order
$>2$ and so the Jacobi sum $J_o$ is a Weil integer of size
$q^{|o|/2}$.

The reader may wonder why we write $\chi_i$ rather than $\chi^i$, and
also why we introduce $e$. The reasons are that $J_o$ is naturally an
element of $\Q(\mu_e)$, and that $\chi_i$ is not in general the $i$-th
power of a character of $\F_{q^{|o|}}^\times$.

\subsection{The $L$-function}
In this subsection, we give an elementary calculation of the
Hasse-Weil $L$-function of $E$ over $\Fq(t)$. 

\begin{thm}\label{thm:L}  Let $O$ be the set of orbits for
  multiplication by $q$ on $\Z/d\Z$ with the orbits $\{0\}$ and
  \textup{(}when $d$ is even\textup{)} $\{d/2\}$ omitted.  For each
  $o\in O$, let $J_o$ be the Jacobi sum defined above, and let $|o|$ be
  the cardinality of $o$.  Then the Hasse-Weil $L$-function of $E$
  over the field $K=\Fq(t^{1/d})$ is
$$L(E/K,T)=\prod_{o\in O}
\left(1-J_o^2T^{|o|}\right).$$ In particular, when $q\equiv1\pmod d$
we have
$$L(E/K,T)=\prod_{\substack{i=1\\i\neq d/2}}^{d-1}
\left(1-J(\lambda,\chi_i)^2T\right).$$
\end{thm}

\begin{proof}
By definition,
$$L(E/K,T)=
\prod_{\text{good  }v}\left(1-a_vT^{\deg(v)}+q_vT^{2\deg(v)}\right)^{-1}
\prod_{\text{bad }v}\left(1-a_vT^{\deg(v)}\right)^{-1}
$$
where the products are over places of $K$ where $E$ has good or bad
reduction, $q_v$ is the cardinality of the residue field at $v$, and
$a_v$ is defined by counting the number of points on the plane cubic
model of $E$ at $v$: $|E(\F_{q_v})|=q_v+1-a_v$.  We note in particular
that $a_v=1,-1,0$ in the case when $E$ has split multiplicative,
non-split multiplicative, or additive reduction respectively.

Expanding $\log L(E/K,T)$ as a series in $T$ and rearranging,
we find that
\begin{equation}\label{eq:lhs1}
\log L(E/K,T)=
\sum_{n=1}^\infty\frac{T^n}{n}\sum_{\beta\in\P^1(\F_{q^n})}a_{\beta,q^n}
\end{equation}
where the sum is over the $\F_{q^n}$-valued points of $\P^1$ and
$a_{\beta,q^n}$ is defined by requiring that $q^n+1-a_{\beta,q^n}$ is the
number of $\F_{q^n}$-valued points on the reduction of $E$ at $\beta$.

Now focus on a particular value of $n$.  We group the points
$\beta\in\P^1(\F_{q^n})$ by their images under the morphism
$\rho:\P^1\to\P^1$, $\beta\mapsto \alpha=\beta^d$ corresponding to the
field extension $\Fq(t^{1/d})/\Fq(t)$.  Let $g=\gcd(q^n-1,d)$ and let
$\psi$ be the character $\psi=t^{(q^n-1)/g}$ of $\F_{q^n}^\times$.
Then the number of $\F_{q^n}$-valued points $\beta$ over a fixed
$\alpha\neq0,\infty$ is either 0 or $g$ depending on whether $\alpha$
is an $g$-th power or not.  This value is equal to the character sum
$\sum_{i=0}^g\psi^i(\alpha)$.  Assuming that $\alpha\neq0,\infty$ so
that $\rho$ is unramified over $\alpha$, we have that
$a_{\beta,q^n}=a_{\alpha,q^n}$.  Thus
$$\sum_{\beta\in\P^1(\F_{q^n})}a_{\beta,q^n}=
a_{0,q^n}+a_{\infty,q^n}
+\sum_{\substack{\alpha\in\P^1(\F_{q^n})\\\alpha\neq0,\infty}}
\sum_{i=0}^{g-1}\psi^i(\alpha)a_{\alpha,q^n}.$$

Now we write $a_{\alpha,q^n}$ as a character sum.  Let $\lambda$ be the
non-trivial character of $\F_{q^n}^\times$ of order 2.  Then by a standard
calculation, for all finite $\alpha$ we have
$$a_{\alpha,q^n}=-\sum_{\gamma\in\F_{q^n}}
\lambda(\gamma(\gamma+1)(\gamma+\alpha)).$$ 

Combining the last two paragraphs, we have
$$\sum_{\beta\in\P^1(\F_{q^n})}a_{\beta,q^n}=a_{\infty,q^n}-\sum_{\alpha\in\F_{q^n}}
\sum_{i=0}^{g-1}\psi^i(\alpha)\sum_{\gamma\in\F_{q^n}}
\lambda(\gamma(\gamma+1)(\gamma+\alpha)).$$

Changing the order of summation and replacing  $\alpha$ with
$\gamma\alpha$, we have 
$$\sum_{\beta\in\P^1(\F_{q^n})}a_{\beta,q^n}=
a_{\infty,q^n}-\sum_{i=0}^{g-1}
\sum_{\gamma\in\F_{q^n}}\lambda(\gamma+1)\psi^i(\gamma)
\sum_{\alpha\in\F_{q^n}}
\lambda(\alpha+1)\psi^i(\alpha).$$

Now since $\lambda(-1)^2=1$, we may change $\lambda(\alpha+1)$ and
$\lambda(\gamma+1)$ to $\lambda(-\alpha-1)$ and $\lambda(-\gamma-1)$.
We find that
$$\sum_{\beta\in\P^1(\F_{q^n})}a_{\beta,q^n}=
a_{\infty,q^n}-\sum_{i=0}^{g-1}
J(\lambda,\psi^i)^2.$$

Finally we note that $J(\lambda,\psi^0)=0$ and, if $g$ is even,
$J(\lambda,\psi^{g/2})^2=1$.  On the other hand, using the reduction
types of $E$ \cite[\S7]{Ulmer14a} and what we said above about
$a_v$ at places of bad reduction, we have that
$$a_{\infty,q^n}=\begin{cases}
1&\text{if $d$ is even}\\
0&\text{if $d$ is odd.}
\end{cases}$$
Thus 
\begin{equation}\label{eq:lhs2}
\sum_{\beta\in\P^1(\F_{q^n})}a_{\beta,q^n}=
-\sum_{\substack{i=1\\i\neq g/2}}^{g-1}
J(\lambda,\psi^i)^2.
\end{equation}

Now we turn to the expression on the right hand side of the formula of
the theorem.  Expanding its $\log$ as a series in $T$ and
rearranging, we have
\begin{equation}\label{eq:rhs1}
\log \prod_{\substack{o\subset\Z/d\Z\\o\neq\{0\},\{d/2\}}}
\left(1-J_o^2T^{|o|}\right)= -\sum_{n=1}^\infty\frac{T^n}n
\sum_{\substack{o\subset\Z/d\Z\\o\neq\{0\},\{d/2\}\\|o|\text{ divides }n}}
J_o^{2n/|o|}|o|
\end{equation}
The union of the orbits which appear here (namely those where $|o|$
divides $n$) is the set of elements $j\in\Z/d\Z$ such that
$(q^n-1)j\equiv0\pmod d$, or equivalently, those divisible by
$e=d/\gcd(q^n-1,d)$.  For such an orbit,
$$J_o^{2n/|o|}|o|=\sum_{j\in o}J\left(\lambda_{q^{|o|}},\chi_j\right)^{2n/|o|}
=\sum_{j\in o}J\left(\lambda_{q^{|o|}},t^{\frac{q^{|o|}-1}dj}\right)^{2n/|o|}
$$
where we write $\lambda_{q^{|o|}}$ to emphasize that in this formula
$\lambda$ is the quadratic character of $\F_{q^{|o|}}$.
Summing over the orbits with $|o|$ dividing $n$, we have
$$\sum_{\substack{o\subset\Z/d\Z\\o\neq\{0\},\{d/2\}\\|o|\text{ divides }n}}
J_o^{2n/|o|}|o|=\sum_{\substack{i=1\\i\neq g/2}}^{g-1}
J\left(\lambda_{q^{|o|}},t^{\frac{q^{|o|}-1}gi}\right)^{2n/|o|}.$$ 
Now the Hasse-Davenport relation \cite[3.7.4]{CohenNT1} shows that
$$J\left(\lambda_{q^{|o|}},t^{\frac{q^{|o|}-1}gi}\right)^{2n/|o|}=
J\left(\lambda_{q^{n}},t^{\frac{q^{n}-1}gi}\right)^{2}=J\left(\lambda_{q^n},\psi^i\right)^2.$$
Thus we have
\begin{equation}\label{eq:rhs2}
\sum_{\substack{o\subset\Z/d\Z\\o\neq\{0\},\{d/2\}\\|o|\text{ divides }n}}
J_o^{2n/|o|}|o|=\sum_{\substack{i=1\\i\neq g/2}}^{g-1}
J\left(\lambda_{q^n},\psi^i\right)^2.
\end{equation}

Comparing equations \eqref{eq:lhs1} and \eqref{eq:lhs2} with
\eqref{eq:rhs1} and \eqref{eq:rhs2} completes the proof of the theorem.
\end{proof}

\section{Explicit Jacobi sums}\label{s:Jacobi}
\numberwithin{equation}{section}
In this section we will make the Jacobi sums of the previous section
sufficiently explicit to calculate the order of vanishing of the
$L$-function at $s=1$ (i.e., at $T=q^{-1}$).

\begin{prop}\label{prop:ExpJacobi}
  Let $o\subset\Z/d\Z$ be an orbit for multiplication by $q$ with
  $o\neq\{0\}$ and $o\neq\{d/2\}$.  Let $e=d/\gcd(d,i)$ for any $i\in
  o$.  Then $J^2_o/q^{|o|}$ is a root of unity if and only if $p$ is
  balanced modulo $e$.  Moreover, if $p$ is balanced modulo $e$ then
  $J^2_o=q^{|o|}$.
\end{prop}

\begin{proof}
  We first show that $J^2_o/q^{|o|}$ is
  a root of unity if and only if $p$ is balanced modulo $e$.  

  Note that $J_o\in\Q(\mu_e)$.  For $a\in(\Z/e\Z)^\times$, let
  $\sigma_a\in\gal(\Q(\mu_e)/\Q)$ be the automorphism with
  $\sigma_a(\zeta_e)=\zeta_e^a$.  Recall that we fixed a prime $\p$ of
  $\Qbar$ over $p$.  We write $\p$ also for the prime of $\Q(\mu_e)$
  that it induces.

  For a rational number $r$, write $\langle r \rangle$ for the
  fractional part of $r$, i.e., the number in $[0,1)$ such that
  $r-\langle r \rangle\in\Z$.

  We set $e=d/\gcd(d,i)$ and $i'=i/\gcd(d,i)$.  Stickelberger's
  theorem (e.g., \cite[Thm.~3.6.6 and Prop.~2.5.14]{CohenNT1}) gives
  the valuation of
$$J_o=J(\lambda,\chi_i)=J\left(t^{(q^f-1)/2},t^{(q^f-1)i/d}\right)$$
at the prime $\sigma_a(\p)$ as
$$-\nu+\sum_{j=0}^{\nu-1}\left\langle\frac{ap^j}{2}\right\rangle
+\left\langle\frac{ai'p^j}{e}\right\rangle
+\left\langle\frac{a(-i'-e/2)p^j}{e}\right\rangle$$ 
where $\nu$ satisfies $q^{|o|}=p^\nu$ and the valuation of $p$ is 1.

Since $J^2_o/q^{|o|}$ is a unit away from primes over $p$, it is a
root of unity if and only if its valuation at every prime over $p$ is
0, or equivalently, if and only the displayed quantity is equal to
$\nu/2$ for all $a$.

Now we note that if $ai'p^j$ has least positive residue modulo $e$ in
the interval $(0,e/2)$, then the sum of the three fractional parts in
the display above is 1; on the other hand, if $ai'p^j$ has least
positive residue modulo $e$ in $(e/2,e)$, then the sum is 2.  Thus the
displayed quantity is $\nu/2$ if and only if exactly half of the
elements of the coset
$ai'\langle p\rangle_e\subset(\Z/e\Z)^\times$ have least positive
residue in $(0,e/2)$ and the other half have least positive residue in
$(e/2,e)$.  This holds for all $a\in(\Z/e\Z)^\times$ if and only if
$p$ is balanced modulo $e$.  This completes the proof that
$J^2_o/q^{|o|}$ is a root of unity if and only if $p$ is balanced
modulo $e$.

Now we assume that $p$ is balanced modulo $e$ and we check that the
root of unity $J^2_o/q^{|o|}$ is in fact 1.  To that end, consider
$\sum_{u+v+1=0}(\lambda(u)-1)(\chi_i(v)-1)$ (cf. \cite[2.5.11
(2)]{CohenNT1}).  Since $\sum_{u}\lambda(u)=\sum_v\chi_i(v)=0$, we
have
$$\sum_{u+v+1=0}(\lambda(u)-1)(\chi_i(v)-1)=J(\lambda,\chi_i)+q^{|o|}.$$
Since $\lambda$ takes values in $\{\pm1\}$, the sum is zero modulo 2.
Thus $J_o=J(\lambda,\chi_i)\equiv1\pmod {2\Z[\mu_e]}$.  But it is easy to
see that the only roots of unity in $\Z[\mu_e]$
congruent to $1\pmod 2$ are $\pm1$.  Thus $J_o=\pm q^{|o|/2}$
and $J_o^2=q^{|o|}$.
\end{proof}



\begin{cor}\label{cor:ord}
Let $K=\Fq(t^{1/d})$.
The order of vanishing of $L(E/K,s)$ at $s=1$ is
$$\sum_{\substack{e|d\\e>2}}\begin{cases}
\frac{\phi(e)}{o_e(q)}&\text{if $p$ is balanced modulo $e$}\\
0&\text{otherwise}
\end{cases}$$
where $\phi$ is Euler's function and $o_e(q)$ is the order of $q$ in
$(\Z/e\Z)^\times$. 
\end{cor}

\begin{proof}
As in Theorem~\ref{thm:L}, let $O$ be the set of orbits
$o\neq\{0\},\{d/2\}$ for multiplication by $q$ on $\Z/d\Z$.  Write
$O_e$ for the subset of orbits $o$ where $\gcd(i,d)=e$ for $i\in
o$, so that $O$ is the disjoint union of the $O_e$ as $e$ runs
through the divisors of $d$.  Then by Theorem~\ref{thm:L}, we have
$$L(E/K,T)=\prod_{e|d}\prod_{o\in O_e}(1-J_o^2T^{|o|}).$$
Since $T=q^{-s}$, for $o\in O_e$ the factor $(1-J_o^2T^{|o|})$
contributes a simple zero at $s=1$ if and only if $p$ is balanced
modulo $e$.  On the other hand, the cardinality of the orbits in $O_e$
is $o_e(q)$, so the number of orbits in $O_e$ is $\phi(e)/o_e(q)$.
Summing over $e$ gives the asserted order of vanishing.
\end{proof}

\begin{proof}[Proof of Theorem~\ref{thm:rank}]
It is proven in \cite[Cor.~11.3]{Ulmer14a} that the Birch and
Swinnerton-Dyer conjecture holds for $E/K$, i.e., we have 
$$\ord_{s=1}L(E/K,s)=\rk E(K).$$
Thus Theorem~\ref{thm:rank} follows immediately from
Corollary~\ref{cor:ord}. 
\end{proof}

\section{Comments on balanced subgroups}\label{s:balanced}
In this section, we make several remarks on the condition that $p$ be
balanced modulo $d$.

\subsection{}
First, we note that the cyclicity of $\langle p\rangle_d$ plays no
role in the definition of balanced.  We may thus define the notion of
a balanced subgroup of $G=(\Z/d\Z)^\times$.  Namely, we say a subgroup
$H$ ``is balanced in $G$'' or ``is balanced modulo $d$'' if for every
coset $gH$ of $H$, the sets $gH\cap A$ and $gH\cap B$
have the same cardinality.

\subsection{}
Next we note that if $H$ and $H'$ are subgroups of $G$ with $H\subset
H'$ and if $H$ is balanced, then so is $H'$.  Indeed, the cosets of
$H'$ are unions of cosets of $H$, so are equally distributed between
$A$ and $B$.

\subsection{}
We call a balanced subgroup ``minimal'' if it does not properly contain
another balanced subgroup.  Since $\{1\}$ is never balanced, a
balanced subgroup of order 2 is automatically minimal.  Examples show
that there can be distinct minimal balanced subgroups for a fixed $d$,
or equivalently, that the intersection of two balanced subgroups need
not be balanced.  For example, if $d=39$, the cyclic subgroups
generated by $7$ and $29$ are balanced, but their intersection (which
is generated by $16$) is not balanced.

\subsection{}
It is clear that $\{\pm1\}\subset\Z/d\Z$ is balanced since $-A=B$.
Therefore, if a subgroup $H$ contains $-1$, then $H$ is balanced.  For
the case $H=\langle p\rangle_d$, this means that $p$ is balanced modulo
$d$ if some power of $p$ is congruent to $-1$ modulo $d$, or
equivalently, if $d$ divides $p^f+1$ for some $f$.  This is the case
that is studied in \cite{Ulmer14a}.

\subsection{}\label{ss:d/2+1}
It is equally clear that if $4|d$ then $\{1,d/2+1\}$ is balanced,
again because $(d/2+1)A=B$.  Therefore $\langle p\rangle_d$ is
balanced when some power of $p$ is congruent to $d/2+1$.  This case
occurs when $p$ is odd and $d=2(p^f-1)$, or more generally when $p$ is
odd, and there is an $f$ such that $d$ divides $2(p^f-1)$ with an odd
quotient.  

Note that if $d$ divides $2(p^f-1)$ with odd quotient, then the same is
true of $d/e$ for every odd $e$ dividing $d$. Thus $p$ is balanced
modulo $d/e$ for every odd $e$ dividing $d$. 
Applying Theorem~\ref{thm:rank} yields the following result.

\begin{prop}\label{prop:d/2+1}
  Suppose that $p$ is odd, $d$ divides $2(p^f-1)$ with an odd
  quotient, $u^d=t$, and $q\equiv1\pmod d$.  Then
$$\rk E(\Fq(u)))=\rk E(\Fq(u^2))+ d/2.$$
\end{prop}

We will study the arithmetic of $E$ over $\Fq(u)$ in
the case when $d=2(p^f-1)$ and $q\equiv1\pmod d$ in more
detail in Sections~\ref{s:points} through \ref{s:sha} below.

\subsection{}
In \cite{PomeranceUlmer13}, it is shown that the case where $4|d$ and
$d/2+1\in\langle p\rangle_d$ is more common than the
``supersingular'' case where $-1\in\langle p\rangle_d$.  More
precisely, for a fixed odd prime $p$, the number of integers $d<X$
satisfying $4|d$ and $d/2+1\in\langle p\rangle_d$ grows faster than
the number of integers $d<X$ with $-1\in\langle p\rangle_d$.

\subsection{}
Machine calculation shows that that there are many pairs $(p,d)$ such
that $p$ is balanced modulo $d$ but $-1\not\in\langle p\rangle_d$ and
$d/2+1\not\in\langle p\rangle_d$.  However, it is conjectured in
\cite{PomeranceUlmer13} and proven in \cite{EngbergThesis} that these
``sporadic'' cases are less common than the case $-1\in\langle
p\rangle_d$.  In other words, for a fixed odd prime $p$, the number of
integers $d<X$ such that $p$ is balanced modulo $d$, but $-1\not\in
\langle p\rangle_d$ and $d/2+1\not\in \langle p\rangle_d$ grows more
slowly than the number of $d<X$ such that $-1\in\langle p\rangle_d$.

\subsection{}
If $H=\langle p\rangle_d$ is balanced, then it has
even cardinality.  In other words, the order of $p$ in
$(\Z/d\Z)^\times$ must be even.

\subsection{}
If $d=\ell^a$ is an odd prime power, then $\langle p\rangle_d$ is
balanced if and only if $-1\in \langle p\rangle_e$.  Indeed, if $p$ is
balanced, then $\langle p\rangle_d$ has even cardinality, so contains
an element of order exactly 2.  But$ (\Z/\ell^a\Z)^\times$ is cyclic
and so contains a unique element of order exactly 2, namely $-1$.
Conversely, we noted above that $p$ is balanced if $-1\in\langle
p\rangle_d$.  The same argument applies when $e=2\ell^a$ is twice an
odd prime power.

\subsection{}
If $p$ is odd and does not divide $d$, then for all
sufficiently large $j$, $p$ is balanced modulo $2^jd$.  Indeed,
suppose that $p^g\equiv1\pmod d$.  If $j$ is large, then $p^g$ has
order $2^e$ modulo $2^j$ for some
$e>1$.  The elements of $(\Z/2^j\Z)^\times$ of order 2 are $-1$,
$2^{j-1}-1$, and $2^{j-1}+1$.  Of these, only $2^{j-1}+1$ is the
square of another element of $(\Z/2^j\Z)^\times$.  This implies that 
$$p^{g2^{e-1}}\equiv 2^{j-1}d+1\pmod{2^jd}$$
and so $p$ is balanced modulo $2^jd$.

Using Theorem~\ref{thm:rank}, we find that for any odd $p$, any power
$q$ of $p$, and any $d$ not divisible by $p$, the rank of $E(K)$ is
unbounded as $K$ runs through the ($2$-adic) tower of fields
$\Fq(t^{1/(2^jd)})$, $j=1,2,\dots$.

\subsection{}
When $d$ is odd, it makes sense to speak of $2$ being balanced
modulo $d$.  We will discuss arithmetic consequences of this case
in Section~\ref{s:p=2} below.

\subsection{}
Examining the first part of the proof of Prop.~\ref{prop:ExpJacobi},
the reader will note that our balanced condition is equivalent to a
certain condition on sums of fractional parts which arises from
considering valuations of Jacobi sums and Stickelberger's theorem.
Similar conditions arise in the study of Fermat varieties, and there
is a large literature on these conditions and the related notion of
``purity of Gauss sums.''  We mention only \cite{ShiodaKatsura79},
\cite{Aoki91}, and \cite{Ulmer07c}.


\section{Explicit points for $d=2(p^f-1)$}\label{s:points}
We observed above that when a power of $p$ is $-1\pmod d$ then
$\langle p\rangle_d$ is balanced.  In \cite[\S3]{Ulmer14a}, explicit
points are exhibited on $E$ when $d=p^f+1$ and $q\equiv1\pmod d$.
Indeed, if $\zeta_d$ denotes a fixed primitive $d$-th root of unity in
$\Fq$ and $u^d=t$, then we have points
$P_i=(\zeta_d^iu,\zeta_d^iu(\zeta_d^iu+1)^{d/2})$ for $i=0,\dots,d-1$
in $E(\Fq(u))$.  We write $V_d$ for the subgroup of $E(\Fq(u))$
generated by the $P_i$.  It is shown in \cite[4.3 and 5.3]{Ulmer14a}
that $V_d$ has finite index in $E(\Fq(t^{1/d}))$.  More generally when
$d$ divides $p^f+1$ for some $f$, the traces down to $E(\Fq(t^{1/d}))$
of the $P_i$ generate a finite index subgroup of the Mordell-Weil
group.

Our goal in this section is to do something similar for $d=2(p^f-1)$.
We work over $\Fq(u)$ where $u^d=t$.  Let
$$R(u)=\left(u^{-2},u^{-3}(u^2+1)^{(p^f+1)/2}\right).$$
A simple calculation shows that $R(u)$ is a rational point on $E$
defined over $\Fp(u)$.  

Next, fix a $d$-th root of unity $\zeta_d$ in $\Fpbar$.  We define
$R_i=R(\zeta_d^iu)$ for $i=0,\dots,d-1$.  Since $E$ is defined over
$\Fp(t)$, the $R_i$ are rational points on $E$ defined over $\Fq(u)$
where $\Fq=\Fp(\mu_d)=\F_{p^{2f}}$.

Note that $\zeta_d^{d/2}=-1$ and that $R_{i+d/2}=-R_i$, so the group
generated by the $R_i$ has rank at most $d/2$.  Because of this
relation we will consider $R_i$ only for $i=0,\dots,d/2-1$.  We write
$W_d$ for the subgroup of $E(\Fq(u))$ generated by $R_i$ for
$i=0,\dots,d/2-1$.

In the next section, we will prove the following results.

\begin{thm}\label{thm:RicardoRank}
  Suppose that $d=2(p^f-1)$, $u^d=t$, and $q\equiv1\pmod d$.  Let
  $W_d$ be the subgroup of $E(\Fq(u))$ generated by the points $R_i$
  defined above.  Then $W_d$ is free abelian of rank $d/2$.  The natural
  homomorphism
$$E(\Fq(u^2))\oplus W_d\to E(\Fq(u))$$
is injective, and its image has finite index. 
\end{thm}

We will show that the index is independent of $q$ and give bounds on
it in Theorem~\ref{thm:index-bounds} below.

Note that if we trace the points $R_i$ down to level $d/2$ (i.e., to
$\Fq(u^2)$), we get zero.  On the other hand, it follows from the
height calculation in the next subsection that if $e$ is a odd divisor
of $p^f+1$ and if we trace from level $d=2(p^f+1)$ to level $d/e$
(i.e., to $\Fq(u^{e})$), the resulting points generate a group of rank
$d/(2e)$.

\section{Heights for $d=2(p^f-1)$}\label{s:heights}
In this section we compute the heights of the points $R_i$ and use
the result to prove Theorem~\ref{thm:RicardoRank}.  As in
\cite[\S8]{Ulmer14a}, we write $\langle\cdot,\cdot\rangle$ for the
canonical height pairing without the factor $\log q$.

\begin{prop}\label{prop:heights}  
  Let $R_i$ be the points exhibited in Section~\ref{s:points}.  For $0\le i,j\le
  d/2-1$, the height pairing has values
$$\langle R_i,R_j\rangle
=\begin{cases}
p^f&\text{if $i=j$}\\
0&\text{if $i\neq j$}.
\end{cases}
$$
If $P\in E(\Fq(u^2))$, then $\langle P,R_i\rangle=0$ for all $i$.
\end{prop}

\begin{proof}
  The general strategy for calculating heights is discussed in
  \cite{Shioda90}.  The case of the Legendre curve and the
  construction of the N\'eron model needed to carry out the
  calculation are discussed in detail in \cite[\S7 and
  \S8]{Ulmer14a}.  Since the case $d=p^f+1$ treated there is
  similar to the case $d=2(p^f-1)$ discussed here, we will not give
  many details.

  Using the Galois invariance of the height pairing and the equality
  $R_{i+d/2}=-R_i$, we may assume that $j=0$.  As in the proof of
  \cite[Thm.~8.2]{Ulmer14a}, we conflate the points $R_i$ and $O$ with the
  corresponding sections of $\pi$.

  We write $\pi:\EE_d\to\P^1$ for the N\'eron model of $E/\Fq(u)$.  We
  recall from \cite[\S7]{Ulmer14a} that $E$ has reduction type
  $I_{2d}$ at $u=0$ and type $I_2$ at the places dividing $u^d-1$.
  Since $d$ is even, the reduction type is $I_{2d}$ at $u=\infty$.
  At all other places, $E$ has good reduction. 

  By \cite[Lemma~7.1]{Ulmer14a}, the height (or degree) of
  $\EE_d\to\P^1$ is $d/2=p^f-1$.  Therefore, the self-intersections
  $O^2$ and $R_i^2$ are $-d/2=-(p^f-1)$.  The intersection number
  $R_0.O$ is equal to 1.  (There is a simple intersection over $u=0$.)
  Therefore, the ``geometric'' part of the height pairing $\langle
  R_0,R_0\rangle$ is
$$-(R_0-O)(R_0-O)=d+2=2p^f.$$

For the ``correction factors,'' one checks that $R_0$ meets the
identity component at $u=0$, the non-identity component at
$u=\zeta_d^j$ where $j=\pm d/4$, the identity component at
$u=\zeta_d^j$ for other values of $j$, and the component labelled $d$
at $u=\infty$. Thus the correction factor is
$$-2\left(\frac{1\cdot1}{2}\right)-\left(\frac{d\cdot d}{2d}\right)=-1-d/2$$
and the height $\langle R_0,R_0\rangle=d/2+1=p^f$.

For $0<i\le d/2-1$, the intersection number $R_i.R_0$ is $2$. (There
are simple intersections over $u=0$ and $u=\infty$.) Thus the
``geometric'' part of the height pairing $\langle R_i,R_0\rangle$ is
$$-(R_i-O)(R_0-O)=-2+2+d/2=d/2.$$

The only place where both $R_i$ and $R_0$ meet a non-identity
component is at $u=\infty$, where they both meet the component labeled
$d$.  The local contribution is thus $-d/2$, and the height pairing is
zero.

This verifies the formula for the heights $\langle R_i,R_j\rangle$.

If $P\in E(\Fq(u^2))$, then $P$ is fixed by $\gal(\Fq(u)/\Fq(u^2))$.
But the automorphism $\sigma$ with $\sigma(u)=-u$ satisfies
$\sigma(R_i)=-R_i$ for all $i$.  Thus 
$$\langle P,R_i\rangle=\langle\sigma(P),\sigma(R_i)\rangle
=-\langle P,R_i\rangle$$ and it follows that $\langle
P,R_i\rangle=0$ for all $i$ as desired.

This completes the proof of the proposition.
\end{proof}

\begin{proof}[Proof of Theorem~\ref{thm:RicardoRank}]
  It follows immediately from Proposition~\ref{prop:heights} that the
  points $R_i$ for $i=0,\dots,d/2-1$ are independent, in other words
  that $W_d$ is free abelian of rank $d/2$.

  If $P$ is a point in the intersection of $W_d$ and $E(\Fq(u))$ then
  by the Proposition we have $\langle P,P\rangle=0$, so that $P$ is
  torsion.  But we just saw that $W_d$ is torsion-free, so $P=0$.
  This proves that $ E(\Fq(u^2))\oplus W_d\to E(\Fq(u))$ is injective.

  By Proposition~\ref{prop:d/2+1}, the rank of $E(\Fq(u))$ is $d/2$ plus the
  rank of $E(\Fq(u^2))$, so the index of $E(\Fq(u^2))\oplus W_d$ in
  $E(\Fq(u))$ is finite.

  This completes the proof of Theorem~\ref{thm:RicardoRank}.
\end{proof}

\section{Bounds on the index for $d=2(p^f-1)$}\label{s:index}
Throughout this section we assume that $d=2(p^f-1)$,
$u^d=t$, and $q\equiv1\pmod d$.  Since $p$ is odd, this implies that
$q\equiv1\pmod4$.  

We begin by recalling several points from \cite[\S3]{Ulmer14a}:
We have 2-torsion points $Q_0=(0,0)$, $Q_1=(-1,0)$, and $Q_t=(-t,0)$,
as well as the 4-torsion points $P^{(2)}_i$ with $i\in\Z/2\Z$ defined
by $P_0^{(2)}=(t^{1/2},t^{1/2}(t^{1/2}+1))$ and
  $P _1 ^{(2)}=(-t^{1/2},-t^{1/2}(-t^{1/2}+1))$. 

Next we recall some results on $\Sel_2(E/\Fq(u))$, the 2-Selmer
group of $E$ over $\Fq(u)$.  This group is defined in
\cite[\S5]{Ulmer14a}, and it is shown there that when $d$ is even
there is an injection
$$\Sel_2(E/\Fq(u))\into(\Z/2\Z)^d$$
and the composed map
$$E(\Fq(u))/2E(\Fq(u))\into\Sel_2(E/\Fq(u))\into(\Z/2\Z)^d$$
sends a point $(x,y)\not\in\{O,Q_1\}$ to the tuple
$(e_0,\dots,e_{d-1})$ where
$$e_j=\ord_{u=\zeta^j}(x+1)\pmod2.$$

It follows immediately from this that if $q'$ is a power of $q$, then
the natural map of 2-Selmer groups
$$\Sel_2(E/\Fq(u))\to\Sel_2(E/\F_{q'}(u))$$
is injective.

There is also an explicit calculation of $\Sel_2(E/\Fq(u))$ in
\cite[\S5]{Ulmer14a}.  In terms of the injection
$\Sel_2(E/\Fq(u))\into(\Z/2\Z)^d$, we have that
$\Sel_2(E/\Fq(u))\cong(\Z/2\Z)^d$ if $q\equiv1\pmod{2d}$, whereas if
$(q-1)/d$ is odd, $\Sel_2(E/\Fq(u))$ consists of those tuples
$(e_0,\dots,e_{d-1})$ satisfying
$$\sum_{i=0}^{d/2-1}e_{2i}=\sum_{i=0}^{d/2-1}e_{2i+1}=0.$$

Recall that $W_d$ is the subgroup of $E(\Fq(u))$ generated by the
points $R_i$, $i=0,\dots,d/2-1$ defined in Section~\ref{s:points}.
Now we turn to the main topic of this section, namely the index of
$E(\Fq(u^2))\oplus W_d$ in $E(\Fq(u))$.

\begin{thm}\label{thm:index-bounds}
  Suppose that $d=2(p^f-1)$, $u^d=t$, and $q\equiv1\pmod d$.  Let $I$
  be the index of $E(\Fq(u^2))\oplus W_d$ in $E(\Fq(u))$.  Then for a
  fixed $p$ and $f$, $I$ is independent of $q$ and is a power of $2$
  times a power of $p$.  The $p$ part of $I$ divides $p^{f(p^f-1)/2}$,
  and $I$ is divisible by 4.
\end{thm}

\begin{proof}
  Note that the index $I$ can only increase with $q$.  We argue that
  it does not increase.  Since the rank of $E(\Fq(u))$ is independent
  of $q$ (for $q$ satisfying our hypotheses), if $I$ increased going
  from $\Fq(u)$ to $\F_{q'}(u)$, there would be a point $P\in
  E(\F_{q'}(u))\setminus E(\Fq(u))$ with $nP\in E(\Fq(u))$.  In this
  case $\Fr_q(P)-P$ would be $n$-torsion.  By
  \cite[Prop.~6.1]{Ulmer14a} we may assume $n=2$.  But the
  injectivity of $\Sel_2(E/\Fq(u))\to\Sel_2(\F_{q'}(u))$ noted above
  implies that a basis of $E(\Fq(u))/2E(\Fq(u))$ is also a basis of
  $E(\F_{q'}(u))/2E(\F_{q'}(u))$, and this implies that the 2 part of
  the index cannot increase.  This proves that $I$ is independent of
  $q$.

  We can also use Selmer groups to see that the index $I$ is divisible
  by 4 as follows.  An easy calculation shows that the image of $W_d$
  in $\Sel_2(\Fq(u))$ has dimension $d/2$ and it contains the image of
  the 4-torsion points $P_0^{(2)}$ and $P _1 ^{(2)}$.  More
  explicitly, we find that
$$\sum_{i=0}^{d/4-1}R_{2i}+P_{1+d/4}^{(2)}\qquad\text{and}\qquad
\sum_{i=0}^{d/4-1}R_{2i+1}+P_{d/4}^{(2)}$$ 
are divisible by 2 in $E(\Fq(u))$.  On the other hand, they are are
not divisible by 2 in $E(\Fq(u))\oplus W_d$ since
Theorem~\ref{thm:RicardoRank} shows the sums $\sum R_{2i}$ and $\sum
R_{2i+1}$ are not divisible by 2 in $W_d$.

To further control the index of $E(\Fq(u^2))\oplus W_d$ in
$E(\Fq(u))$, we use a relative version of the integrality result
\cite[Prop~9.1]{Ulmer14a}.  To state the result, we first
introduce some notation.  If $M$ is a finitely generated $\Z$-module,
we write $M[1/2]$ for $M\tensor\Z[1/2]$.  If $M$ is also equipped with
an automorphism of order 2, we have a direct sum decomposition
$$M[1/2]=M^+\oplus M^-$$ 
where $M^+$ denotes the fixed subgroup of $M[1/2]$ and $M^-$
denotes the subgroup of $M[1/2]$ where the automorphism acts as
$-1$.  If $M$ is equipped with a symmetric, bilinear, $\Q$-valued
pairing, we may canonically extend the pairing to $M[1/2]$.  If $M$
(resp. $M'$) is a finitely generated $\Z$-module (resp.~a finitely
generated $\Z[1/2]$-module) equipped with a symmetric, bilinear,
$\Q$-valued pairing, we define $\disc(M)$ and $\disc(M')$ in the usual
way: choose a basis of $M$ modulo torsion (resp.~a basis of $M'$
modulo $\Z[1/2]$-torsion) and let $\disc$ be the absolute value of the
determinant of the matrix of pairings of basis elements.  Then
$\disc(M)$ is a well-defined element of $\Q$, and $\disc(M')$ is well
defined up to multiplication by the square of an element of
$\Z[1/2]^\times$, i.e., up to a power of $4$.

Now returning to elliptic curves, we let $\pi:\EE\to\P^1$ be the
N\'eron model of $E/\Fq(u)$, and we let $N$ be the subgroup of the
N\'eron-Severi group of $\EE$ generated by components of fibers of
$\pi$ that do not meet the zero section.  This is known to be a
finitely generated free abelian group.  With these notations,
\cite[Prop.~9.1]{Ulmer14a} says that the rational number
$$\frac{\disc(E(\Fq(u)))\disc(N)}{|E(\Fq(u))_{tor}|^2}$$
is in fact an integer.  If we take the proof of this result, tensor
all the groups appearing in it with $\Z[1/2]$, and then take the minus
part for $\gal(\Fq(u)/\Fq(u^2))=\{1,\sigma\}$, we find that the
rational number
$$\frac{\disc(E(\Fq(u))^-)\disc(N^-)}{|E(\Fq(u))^-_{tor}|^2}$$
(well-defined up to squares of elements of $\Z[1/2]^\times$) is in
fact an element of $\Z[1/2]$.

Since $E(\Fq(u))_{tor}$ is a 2-group, $E(\Fq(u))^-_{tor}$ is trivial.  Using
the N\'eron model calculations of \cite[\S7]{Ulmer14a}, we find
that all components of the fibers of $\pi$ over $0$ and $\infty$ are
fixed by $\sigma$ and so do not contribute to $N^-$.  The fibers over
the $d$-th roots of unity are permuted in pairs, so $N^-$ is a free
$\Z[1/2]$-module of rank $d/2$, and it has an orthogonal set of
generators each of which has self-pairing $-4$.  Thus
$\disc(N^-)=(-4)^{d/2}=2^d$ which is a unit in $\Z[1/2]$.  Our
integrality result then yields that $\disc(E(\Fq(u))^-)$ lies in
$\Z[1/2]$.

Now the inclusion $E(\Fq(u^2))\oplus W_d\into E(\Fq(u))$ yields an
identification
$$\left(\frac{E(\Fq(u))}{E(\Fq(u^2))\oplus W_d}\right)[1/2]\cong
\frac{E(\Fq(u))^-}{W_d[1/2]}.$$
If $J$ denotes the order of this group (which is the prime-to-2 part
of the index of $E(\Fq(u^2))\oplus W_d$ in $E(\Fq(u))$), then we have
$$\disc(E(\Fq(u))^-)=\frac{\disc(W_d[1/2])}{J^2}.$$
Proposition~\ref{prop:heights} shows that $\disc(W_d[1/2])$ is
$p^{f(p^f-1)}$, so our integrality result shows that $J$ is a power of
$p$.  Therefore, the index of $E(\Fq(u^2))\oplus W_d$ in $E(\Fq(u))$
is a power of $2$ times a power of $p$.  Moreover, the power of $p$
divides $p^{f(p^f-1)/2}$.
\end{proof}

It follows from Theorem~\ref{thm:sha-bounds} below that the 2 part of
$I$ divides $2^{p^f+1}$.  We suspect that the 2 part of $I$ is in fact
always 4.  We can prove this when $p^f\equiv3\pmod4$ by comparing the
trace down to level $d=4$ of points $R_i$ of this paper and the points
$P_i$ of \cite[\S3]{Ulmer14a}.  It is also true in several other
examples we have checked, but we do not know how to prove it in
general.

\section{Bounds on $\sha$ for $d=2(p^f-1)$}\label{s:sha}
Recall that the second part of the BSD conjecture relates the leading
coefficient of the $L$-series of $E$ at $s=1$ to other invariants of
$E$.  More precisely, defining $L^*(E):=\frac1{r!}L^{(r)}(E,1)$, we
should have
$$L^*(E)=\frac{|\sha|\,R\,\tau}{|E_{tor}|^2}$$
where $r$ is the order of vanishing of the $L$-function, $\sha$ is the
Tate-Shafarevich group of $E$, $R$ is a regulator, and $\tau$ is a
Tamagawa number.  See \cite[\S6]{UlmerCRM} for definitions and a
precise statement.  The conjecture holds in our context by
\cite[Cor.~11.3]{Ulmer14a}.

Our goal in this section is to exploit the BSD formula to obtain
information on $\sha(E/\Fq(u))$ when $d=2(p^f-1)$.  In
\cite[Cor.~10.2]{Ulmer14a} it is proven that when $d=p^f+1$,
$\sha(E/\Fq(u))$ is a $p$-group of order equal to the square of an
index.  The analogue for $d=2(p^f-1)$ is necessarily more complicated
because we have very little information in general on the arithmetic
of $E$ over $\Fq(u^2)$.  Nevertheless, we can prove a strong relative
statement.

To state the theorem, we write 
$$\NN:\sha(E/\Fq(u))\to\sha(E/\Fq(u^2))$$
for the norm map induced by the corestriction $H^1(\spec\Fq(u),E)\to
H^1(\spec\Fq(u^2),E)$. 

\begin{thm}\label{thm:sha-bounds}
Suppose that $d=2(p^f-1)$, $u^d=t$, and $q\equiv1\pmod d$.  Let $I$ be
the index of $E(\Fq(u^2))\oplus W_d$ in $E(\Fq(u))$, as in
Theorem~\ref{thm:index-bounds}. 
\begin{enumerate}
\item We have
$$I^2=2^{p^f+1}\frac{p^{fd/2}}{q^{d/4}}
\frac{|\sha(E/\Fq(u))|}{|\sha(E/\Fq(u^2))|}
=2^{p^f+1}\frac{p^{fd/2}}{q^{d/4}}
\frac{|\Ker\NN\,|}{|\coker\NN\,|}.$$ 
In particular, when $q=p^{2f}$, we have
$$I^2=2^{p^f+1}\frac{|\sha(E/\Fq(u))|}{|\sha(E/\Fq(u^2))|}
=2^{p^f+1}\frac{|\Ker\NN\,|}{|\coker\NN\,|}.$$
\item The cokernel of $\NN$ is a 2-group.  
\item The kernel of $\NN$ is a $p$-group.
\item For fixed $p$ and $f$, the cokernel of $\NN$ is independent of
  $q$.  The cokernel of $\NN$ is non-trivial if
  $p^f\equiv1\pmod4$.  The order of $\Ker\NN$ goes to infinity with
  $\log_p q$.
\end{enumerate}
\end{thm}

\begin{proof}
  By Theorem~\ref{thm:L}, Proposition~\ref{prop:ExpJacobi}, and
  Subsection~\ref{ss:d/2+1}, we have
$$\frac{L(E/\Fq(u),s)}{L(E/\Fq(u^2),s)}=(1-q^{1-s})^{d/2}$$
and so $L^*(E/\Fq(u))=L^*(E/\Fq(u^2))$.  Taking the ratio of the
two BSD formulas, we find
$$1=\frac{R(E/\Fq(u))}{R(E/\Fq(u^2))}\frac{|\sha(E/\Fq(u))|}{|\sha(E/\Fq(u^2))|}
\frac{\tau(E/\Fq(u))}{\tau(E/\Fq(u^2))}.$$

Using the notation of Section~\ref{s:index}
and Theorem~\ref{thm:RicardoRank}, we have
$$R(E/\Fq(u))=\disc(E(\Fq(u)))=
\disc\left(E(\Fq(u^2))\oplus W_d\right)/I^2.$$
On the other hand, $E(\Fq(u^2))\oplus W_d$ is
an orthogonal direct sum and $\disc(W_d)=p^{fd/2}$ by 
Prop.~\ref{prop:heights}.  Thus 
$$\frac{R(E/\Fq(u))}{R(E/\Fq(u^2))}=\frac{p^{fd/2}}{I^2}.$$

For the $\tau$ ratio, we have 
$\tau(E/\Fq(u))=2^{d+2}d^2q^{1-d/2}$
and $\tau(E/\Fq(u^2))=2^{d/2}d^2q^{1-d/4}$, so
$$\frac{\tau(E/\Fq(u))}{\tau(E/\Fq(u^2))}=2^{d/2+2}q^{-d/4}=2^{p^f+1}q^{-d/4}.$$

Substituting the last two displays into the ratio of BSD formulas
yields part (1).  

We saw in Theorem~\ref{thm:RicardoRank} that $I$ is a power of $2$
times a power of $p$, so part (1)  shows that the same is true of 
$|\sha(E/\Fq(u))|/|\sha(E/\Fq(u^2))|$.

Next we consider the pull-back and norm maps on the Tate-Shafarevich
groups induced by the inclusion of fields $\Fq(u^2)\into\Fq(u)$:
$$\sha(E/\Fq(u^2))\labeledto{j}\sha(E/\Fq(u))\labeledto{\NN}\sha(E/\Fq(u^2)).$$
The composition $\NN\compose j$ is multiplication by $2$, so is an
isomorphism on the prime-to-2 part of $\sha(E/\Fq(u^2))$.  Therefore,
$\NN$ is surjective on the prime-to-2 parts.  This proves
part (2), namely that $\coker\NN$ is a 2-group.  Bounds on its
order will be proven below.

Since the prime-to-$2p$ parts of $\sha(E/\Fq(u))$ and
$\sha(E/\Fq(u^2))$ have the same order, they are isomorphic via $\NN$.
We will prove that the kernel of $\NN$ has odd order and
that the cokernel is non-trivial if $(q-1)/d$ is odd.  For these
assertions, we consider the diagram of descent sequences:
$$\xymatrix{
0\ar[r]&E(\Fq(u))/2E(\Fq(u))\ar[r]\ar[d]&\Sel_2(E(\Fq(u)))\ar[r]\ar[d]&
\sha(E/\Fq(u))_2\ar[r]\ar[d]&0\\
0\ar[r]&E(\Fq(u^2))/2E(\Fq(u^2))\ar[r]&\Sel_2(E(\Fq(u^2)))\ar[r]&
\sha(E/\Fq(u^2))_2\ar[r]&0}$$
where the vertical maps are norms.  

Suppose that $(q-1)/d$ is even.  The Selmer group calculations
recalled above show that if $r$ is the rank of $E(\Fq(u^2))$, then the
diagram above is isomorphic to
$$\xymatrix{
  0\ar[r]&\F_2^{d/2+r+2}\ar[r]\ar[d]&\F_2^d\ar[r]\ar[d]&
  \sha(E/\Fq(u))_2\ar[r]\ar[d]&0\\
  0\ar[r]&\F_2^{r+2}\ar[r]&\F_2^{d/2}\ar[r]&
  \sha(E/\Fq(u^2))_2\ar[r]&0}$$ and the middle vertical map is
surjective.  It follows that the right vertical map is also
surjective.  On the other hand, considering the rows shows that the
right hand groups have the same order, so they are isomorphic via the
vertical map, in other words
$$\NN:\sha(E/\Fq(u))_2\to \sha(E/\Fq(u^2))_2$$  
is an isomorphism.  A chase using the diagram
$$\xymatrix{
0\ar[r]&\sha(E/\Fq(u))_2\ar[r]\ar[d]&\sha(E/\Fq(u))_{2^n}\ar^2[r]\ar[d]
&\sha(E/\Fq(u))_{2^{n-1}}\ar[d]\\
0\ar[r]&\sha(E/\Fq(u^2))_2\ar[r]&\sha(E/\Fq(u^2))_{2^n}\ar^2[r]
&\sha(E/\Fq(u^2))_{2^{n-1}}}$$
and induction on $n$ shows that
$$\NN:\sha(E/\Fq(u))_{2^\infty}\to \sha(E/\Fq(u^2))_{2^\infty}$$  
is injective.

Matters are slightly more complicated when we assume that $(q-1)/d$ is
odd.  In this case, the Selmer group calculations recalled above show
that that the descent diagram is isomorphic to
$$\xymatrix{
  0\ar[r]&\F_2^{d/2+r+2}\ar[r]\ar[d]&\F_2^{d-2}\ar[r]\ar[d]&
  \sha(E/\Fq(u))_2\ar[r]\ar[d]&0\\
  0\ar[r]&\F_2^{r+2}\ar[r]&\F_2^{d/2}\ar[r]&
  \sha(E/\Fq(u^2))_2\ar[r]&0}$$ 
and the middle vertical map has cokernel of order 4.  This shows that
the cokernel of the right vertical map has order at most 4.  On the
other hand, considering the rows shows that the orders of the right
hand groups have ratio 4, and we conclude that the right hand vertical
map is injective.  The same argument as in the case $(q-1)/d$ is even
then shows that
$$\NN:\sha(E/\Fq(u))_{2^\infty}\to \sha(E/\Fq(u^2))_{2^\infty}$$  
is injective.  If it were surjective, it would be an isomorphism, and
therefore induce an isomorphism on the 2-torsion subgroups, but we just
saw this is not the case, so the cokernel is non-trivial.  

Summarizing the last three paragraphs, we have that the 2 part of
$\ker\NN$ is always trivial, and the 2 part of $\coker\NN$ is
non-trivial when $(q-1)/d$ is odd.  This completes the proof that
$\Ker\NN$ is a $p$-group, i.e., part (3).

We saw in Theorem~\ref{thm:index-bounds} that for fixed $p$ and $f$,
$I$ is independent of $q$.  Since the kernel and cokernel of $\NN$ are
$p$-groups and $2$-groups respectively, it follows from parts (1),
(2), and (3) that the cokernel of $\NN$ is independent of $q$, and the
order of the kernel goes to infinity with $q$.  Finally, if
$p^f\equiv1\pmod 4$ and we take $q=p^{2f}$, then $(q-1)/d=(p^f+1)/2$
is odd and the cokernel of $\NN$ is non-trivial.  This completes the
proof of part (4) and the proof of the theorem.
\end{proof}

\section{Correspondences and points}\label{s:Berger}
It is proven in \cite[\S11]{Ulmer14a} that the N\'eron model
$\EE_d\to\P^1$ of $E/\Fq(u)$ is birational to the quotient of a
product of curves.  This implies the BSD conjecture for $E/\Fq(u)$ and
it allows us to compute the rank of $E(\Fq(u))$ in terms of Jacobians
over finite fields.  It also allows us to ``explain'' the explicit
points of this paper and \cite{Ulmer14a} via certain correspondences.
In this section we briefly explain how this plays out for $d=p^f+1$
and $d=2(p^f-1)$.

Let $\EE_d$ be the N\'eron model of $E/\Fq(u)$, so the function field of
$\EE_d$ is generated by $x$, $y$, $u$ with relation
$y^2=x(x+1)(x+u^d)$.

Let $\CC_d$ be the smooth projective curve with affine model
$$z^d+x^2+1=0$$  and let $\DD_d$ be defined by 
$$w^d+y^2+1=0.$$  
There is an action of $\mu_d\times\mu_2$ on $\CC_d\times\DD_d$ given
by
$$\zeta_d(x,y,z,w)=(x,y,\zeta_dz,\zeta_d^{-1}w)\qquad
(-1)(x,y,z,w)=(-x,-y,z,w).$$

The quotient has function field generated by $x^2$, $xy$, and $zw$.
It is birational to $\EE_d$ with the quotient map
$\phi:\CC_d\times\DD_d\to\EE_d$ given by
$$\phi^*(x,y,u)=(-x^2-1,xy(x^2+1),zw).$$

The rational quotient map $\CC_d\times\DD_d\ratto\EE_d$ can be used to
compute the zeta-function of $\EE_d$ and the $L$-function of
$E/\Fq(u)$.  The fact that $\EE_d$ is dominated by a product of
quotients of Fermat curves ``explains'' why the $L$-function of $E$ is
expressible in terms of the Jacobi sums
$J(\lambda,\chi_i)=J(\chi_{d/2},\chi_i)$.

Let $J=J_{\CC_d}$ be the Jacobian of $\CC_d$.  The results of
\cite{Ulmer13a} use the rational map $\phi$ to prove that the rank of
$E(\Fpbar(t^{1/d})$ is equal to the rank of the group of endomorphisms
of $J$ which anti-commute with the action of $\mu_d$.  This group can
also be identified with a certain group of correspondences on
$\CC_d\times\DD_d$.

The geometry of this set-up can be used to ``explain'' the explicit
points of \cite{Ulmer14a} and this paper.  Namely, when $d=p^f+1$,
consider the graph of the $p^f$-power Frobenius
$\Fr_{p^f}:\CC_d\to\DD_d$.  This anti-commutes with the $\mu_d$
action.  Pushing it down via $\phi$ yields the section corresponding to
the point $P_0$ of \cite[Thm.~8.1]{Ulmer13a}.

When $d$ is even, we have an automorphism $\psi:\CC_d\to\CC_d$ given
by 
$$\psi^*(x,z)=(x/z^{d/2},z^{-1}).$$
When $d=2(p^f-1)$, the composition of $\psi$ and the $p^f$-power
Frobenius gives a morphism $\CC_d\to\DD_d$ which induces a
homomorphism $J_{\CC_d}\to J_{\DD_d}$ which anti-commutes with the
$\mu_d$ actions.  Pushing the graph of this morphism down to $\EE_d$
via $\phi$ leads to the section associated to the point $R_0$ of
Section~\ref{s:points}.

Note that Theorem~\ref{thm:rank} shows that $E$ has points of infinite
order over many extensions $\Fq(t^{1/d})$, but we lack any explicit
expression for their coordinates except when $d$ divides $p^f+1$ or
$2(p^f-1)$.  It would be very interesting to find a direct connection
between the balanced condition and special endomorphisms or
correspondences on $\CC_d\times\DD_d$ which would provide explicit
expressions for points of infinite order on $E$.

\section{The case $p=2$}\label{s:p=2}
The equation for the Legendre curve does not define an elliptic curve
in characteristic 2.  However, there is a curve $E'$ defined uniformly
for all $p$ which is isogenous to the Legendre curve for all $p>2$.
Most of the main results of this paper extend to $E'$ and continue to
hold in characteristic 2.

More precisely, let $t'$ and $u'$ be new indeterminates with
$u^{\prime d}=t'$, and let $E'$
be the elliptic curve
\begin{equation}\label{eq:E'}
y^{\prime2}+x'y'+t'y'=x^{\prime 3}+t'x^{\prime 2}
\end{equation}
over $\Fp(t')$.  It is proven in \cite[Lemma 11.1]{Ulmer14a} that if
$p>2$ and we identify $\Fp(t')$ and $\Fp(t)$ by sending $t'$ to
$t/16$, then $E$ and $E'$ are 2-isogenous.  If $16$ is a $d$-th power
in $\Fq$, then we may identify $\Fq(u)$ and $\Fq(u')$ (as extensions
of $\Fp(t)=\Fp(t')$) and in this case $E$ and $E'$ have the same
$L$-function over $\Fq(u)$ and the same rank because they are
isogenous.  We would like to make a more precise statement that holds
over $\Fq(u')$ for all $q$.

To that end, we define a new Jacobi sum $J'_o$ where $o\subset\Z/d\Z$ is
an orbit for multiplication by $q$.  As before, we assume that
$o\neq\{0\}$ and $o\neq\{d/2\}$ if $d$ is even.  We keep the notations of
Subsection~\ref{ss:Jacobi} and define
$$J'_o=J(\chi_i,\chi_i)$$
where $i$ is any element of $o$.  This is well defined because
$J(\chi_{qi},\chi_{qi})=J(\chi_i,\chi_i)$.   

The following theorem extends many of our results about $E$ to $E'$,
where they continue to hold when $p=2$.

\begin{thm}
Let $p$ be an arbitrary prime number, $q$ a power of $p$, and $d$ an
integer prime to $p$.  Let $K'=\Fq(u')$ where $u^{\prime d}=t'$, and
let $E'$ be the elliptic curve over $K'$ defined by \eqref{eq:E'}.
\begin{enumerate}
\item The Hasse-Weil $L$-function of $E'$ over $K'$ is
$$L(E'/K',T)=\prod_{o\in O}
\left(1-J^{\prime2}_oT^{|o|}\right).$$ 
Here the product is over the set $O$ of orbits $o\subset\Z/d\Z$ for
multiplication by $q$, omitting the orbits $\{0\}$ and $\{d/2\}$
\textup{(}if $d$ is even\textup{)}.
\item Let $o\in O$ be an orbit and let $e=d/\gcd(d,i)$ for any $i\in
  o$. Then $J^{\prime 2}_o/q^{|o|}$ is a root of unity if and only if
  $p$ is balanced modulo $e$.  
\item If $p>2$ and $i\in o$, then $J'_o=\chi^i(4)\lambda(-1)J_o$.  
\item Suppose $p=2$ and $o\in O$.  If $e=d/\gcd(d,i)$ for
  any $i\in o$ and $2$ is balanced modulo $e$, then $J^{\prime
    2}_o=q^{|o|}$.  If $p=2$, then the rank of $E'(\Fq(u'))$ is
$$\sum_{\substack{e|d\\e>2}}\begin{cases}
\frac{\phi(e)}{o_e(q)}&\text{if $2$ is balanced modulo $e$}\\
0&\text{otherwise}\end{cases}$$
where $\phi$ is Euler's function and $o_e(q)$ is the order of $q$ in
$(\Z/e\Z)^\times$. 
\item Let $\EE'_d\to\P^1$ be the N\'eron model of $E'/\Fq(u')$.  Let
  $\CC'_d=\DD'_d$ be the smooth, projective \textup{(}Fermat\textup{)}
  curve over $\Fq$ defined by $z^d=x(1-x)$.  Then $\EE_d$ is birational
  to the quotient of $\CC_d\times\DD_d$  by the anti-diagonal action
  of $\mu_d$.
\end{enumerate}
\end{thm}

\begin{proof}
The proofs are for the most part parallel to those of the
corresponding statements for $E$, so we will omit many details.

For (1), one may give a direct, elementary argument along the lines of
the proof of Theorem~\ref{thm:L}.  Alternatively, one may use
part (5) of this theorem and a geometric analysis of the morphism
$(\CC'_d\times\DD'_d)/\mu_d\ratto\EE'_d$.  Indeed, the inverse roots of the zeta
function of $\CC'_d$ are exactly the Jacobi sums $J'_o$ and the
inverse roots of the $H^2$ part of the zeta function of
$(\CC'_d\times\DD'_d)/\mu_d$ are the squares $J^{\prime 2}_o$.  In some
sense this ``explains'' their appearance in the $L$-function of $E'$.

The proof of (2) is parallel to the first part of the proof of
Proposition~\ref{prop:ExpJacobi}.  Indeed, Stickelberger's theorem shows
that the valuation of  $J'_o$ at the prime $\sigma_a(\p)$ is given by
$$-\nu+\sum_{j=0}^{\nu-1}\left\langle\frac{ai'p^j}{e}\right\rangle
+\left\langle\frac{ai'p^j}{e}\right\rangle
+\left\langle\frac{-2ai'p^j}{e}\right\rangle$$ 
where $i'=i/\gcd(d,i)$, $\nu$ satisfies $q^{|o|}=p^\nu$, and the
valuation of $p$ is 1.  It is then easy to see that the valuations are
all $|o|/2$ if and only if $p$ is balanced modulo $e$.

Part (3) is \cite[Exer.~40 in Chap.~2]{CohenNT1}.  

For part (4), if $p=2$ and $2$ is balanced modulo $e$, then by part
(2), $J^{\prime 2}_o/q^{|o|}$ is a root of unity.  To see that it is
1, we may use Stickelberger's congruence \cite[Thm.~3.6.6]{CohenNT1}
to see that $J'_o/q^{|o|/2}\equiv 1$ modulo any prime of $\Q(\mu_e)$
over $2$, so that $J'_o/q^{|o|/2}=\pm1$.  The rest of part (4) then
follows from parts (1) and (2).

For (5), let $\CC'_d$ and $\DD'_d$ be given by $z^d=x(1-x)$ and
$w^d=y(1-y)$, so that the function field of the quotient
$(\CC'_d\times\DD'_d)/\mu_d$ is generated by $x$, $y$, and $zw$.  The
function field of $\EE'_d$ is generated by $u'$, $x'$ and $y'$ with relation
$y^{\prime2}+x'y'+t'y'=x^{\prime 3}+u^{\prime d}x^{\prime 2}$.   It is
easy to see that the assignments $u'\mapsto
zw$, $x'\mapsto -(zw)^d/y$, and $y'\mapsto(zw)^dx(1-y)/y$ yield a
well-defined isomorphism between the two function fields.  (This is
essentially the example discussed in \cite[\S7]{Ulmer13a}.)
\end{proof}

We note that part (3) is an arithmetic reflection of the fact that
the Fermat curve $\CC'_d$ in part (5) is a twisted form of the Fermat
curve $\CC_d$ of the previous section.  It follows from part (3) that
$J^{\prime 2}_o=\chi^i(16)J^2_o$.  This relation between the inverse
roots of the $L$-functions of $E$ and $E'$ is an arithmetic
reflection of the identification $\Fq(u)=\Fq(u')$ and the resulting
isogeny between $E$ and $E'$ when $16$ is a $d$-th power in
$\Fq^\times$.

We also note that if $p>2$, $d>2$ divides $p^f+1$, and $q\equiv1\pmod
d$, then $16$ (and indeed any integer) is a $d$-th power in $\Fq$.
Thus the rank result of \cite[Thm.~7.5]{Ulmer13a} can be recovered
from the theorem.

\bibliography{database}{}
\bibliographystyle{alpha}

\end{document}